\documentclass[11pt]{amsart}
\usepackage[top=3cm,bottom=2cm,left=2cm,right=2cm]{geometry}
\usepackage{indentfirst}
\usepackage[english]{babel}

\usepackage{hyperref}
\hypersetup{
	colorlinks,
	linkcolor={MidnightBlue},
	citecolor={Green}
}

\usepackage{amsfonts,amsthm,latexsym,amsmath,amssymb,amscd,amsmath, mathrsfs, epsf,bm}
\usepackage[utf8]{inputenc}
\usepackage[dvipsnames]{xcolor}
\usepackage{tikz}
\usepackage{parskip}
\usetikzlibrary{positioning, arrows.meta, 3d}
\usetikzlibrary{arrows,matrix}
\usepackage[nameinlink,noabbrev]{cleveref}
\usepackage{graphicx}
\usepackage[font = footnotesize]{caption}
\usepackage{mlalgebra}
\usepackage{enumitem}
\usepackage{booktabs}
\usepackage{comment}

\usepackage[dvipsnames]{xcolor}
\definecolor{commentgreen}{RGB}{2,112,10}
\definecolor{eminence}{RGB}{108,48,130}

\definecolor{frenchplum}{RGB}{129,20,83}

\usepackage{listings}
\lstdefinelanguage{code}{
basicstyle=\small\ttfamily,
alsoletter=",
classoffset=1,
keywords={gb, eliminate, saturate, diff, degree, flatten, apply, tensor, product},
keywordstyle={\color{teal}},
classoffset=2,
morekeywords={from, to, list, terms, toList, entries, for, end, if, return},
keywordstyle={\color{blue}},
classoffset=3,
morekeywords={QQ},
keywordstyle={\color{frenchplum}},
classoffset=4,
morekeywords={ideal, matrix, gens},
keywordstyle={\color{teal}},
xleftmargin=1.5cm,
xrightmargin=1em,
columns=fullflexible,
keepspaces=true,
stepnumber=1,
numbers=none,
captionpos=b,
showspaces=false,
frame=none
}

\definecolor{weborange}{RGB}{255,165,0}

\definecolor{darkgray}{rgb}{.4,.4,.4}
\usepackage{mathtools}

\newtheorem{theorem}{Theorem}
\numberwithin{theorem}{section}

\newtheorem{lemma}[theorem]{Lemma}
\newtheorem{corollary}[theorem]{Corollary}

\theoremstyle{definition}

\newtheorem{example}[theorem]{Example}
\newtheorem*{claim*}{\indent Claim}

\newcommand{\PP}{\mathbb{P}}
\newcommand{\CC}{\mathbb{C}}
\newcommand{\RR}{\mathbb{R}}
\newcommand{\ZZ}{\mathbb{Z}}

\DeclareMathOperator*{\mean}{\operatorname{\mathbb E}}

\newcommand{\U}{\mathop{\rm U}\nolimits}
\newcommand{\rk}{\mathop{\rm rk}\nolimits}

\newcommand{\Gr}{\mathop{\rm G}\nolimits}

\newcommand{\mlrk}{\mathop{\rm mlrank}\nolimits}
\newcommand{\Unif}{\mathop{\rm Unif}\nolimits}

\newcommand{\vol}{\mathop{\rm vol}}

\newcommand{\im}{\mathrm{im}}
\newcommand{\iu}{\mathrm{i}}

\newcommand{\St}{\mathop{\rm St}\nolimits}

\newcommand{\B}{\mathcal{B}}

\newcommand{\C}{C}

\newcommand{\NJ}{\mathop{\rm NJ}}

\newcommand{\bfn}{\mathbf{n}}
\newcommand{\bfk}{\mathbf{k}}
\newcommand{\bfa}{\mathbf{a}}
\newcommand{\bfe}{\mathbf{e}}
\newcommand{\bfx}{\mathbf{x}}
\newcommand{\bfy}{\mathbf{y}}
\newcommand{\bfu}{\mathbf{u}}
\newcommand{\bfv}{\mathbf{v}}
\newcommand{\bfw}{\mathbf{w}}
\newcommand{\bfp}{\mathbf{p}}
\newcommand{\bfb}{\mathbf{b}}
\newcommand{\bfc}{\mathbf{c}}

\title{Degree of the Subspace Variety}

\author[Breiding]{Paul Breiding}
\address{Paul Breiding\\ Osnabr\"uck University, Germany}
\email{pbreiding@uni-osnabrueck.de}

\author[Santarsiero]{Pierpaola Santarsiero}
\address{
 Pierpaola Santarsiero\\
Osnabr\"uck University, Germany
}
\email{pierpaola.santarsiero@uni-osnabrueck.de}

\begin{document}

\begin{abstract}
Subspace varieties are algebraic varieties whose elements are tensors with bounded multilinear rank. In this paper, we compute their degrees by computing their volumes.
\end{abstract}

\maketitle

\section{Introduction}
 
Among the many notions of tensor rank, perhaps the most elementary one is the \emph{multilinear rank}, which only relies on linear algebra concepts and dates back to the work of Hitchcock \cite{hit27}. Given a $d$-order tensor $T\in \CC^{n_1}\otimes \cdots \otimes \CC^{n_d}$, the \emph{$i$-th flattening} $T^{(i)}$ of $T$ is the linear map $T^{(i)}\colon  \bigotimes_{j=1,j\neq i}^d \CC^{n_j*}\rightarrow \CC^{n_i}$ defined by contraction. The multilinear rank of $T\in \CC^{n_1}\otimes \cdots \otimes \CC^{n_d}$ is the $d$-uple containing the rank of all flattenings of $T$ and it is usually denoted $\mlrk(T):=(\rk T^{(1)},\dots, \rk T^{(d)})$.
The multilinear rank provides the \emph{Tucker decomposition} \cite{tucker}. Given a tensor $T\in \CC^{n_1}\otimes \cdots \otimes \CC^{n_d}$ of multilinear rank $\bfk = (k_1,\dots,k_d)$, the Tucker decomposition of $T$ factorizes the tensor through a \emph{core tensor} $\C\in \CC^{k_1}\otimes \cdots \otimes \CC^{k_d}$ and $d$ matrices~$A_i\in \CC^{n_i\times k_i}$ of maximal rank as
$T=(A_1\otimes \cdots \otimes A_d)\cdot \C.$

\smallskip
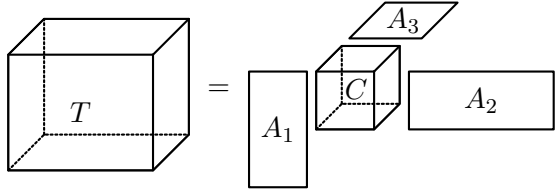
\begin{figure}[h]
	\centering
 \begin{tikzpicture}
\draw (0,0) node{$\mlatensor[$T$]{4,5,3.5} \ = \  \mlarankLMNtensor[$A_1$,$A_2$,$A_3$]{4,5,3.5}{2,2,2.5} $};
\draw (1,0.075) node{$C$};
\end{tikzpicture}
 \caption{Tucker decomposition $T = (A_1 \otimes A_2\otimes A_3)\cdot \C$ for a $3$-order tensor $T\in\mathbb C^{n_1}\otimes \CC^{n_2}\otimes \CC^{n_3}$.}
	\label{tucker image}
\end{figure}
\smallskip
The usual notion of tensor rank is invariant under this operation, and the multilinear rank gives a bound on the usual tensor rank (see, e.g., \cite{Lands}).  The process of decomposing a tensor via the Tucker decomposition is also known in the algebraic literature as \emph{concision process} (see, e.g., \cite{Lands}). 

The Tucker decomposition is useful for dimensionality reduction and it allows us to highlight a core tensor~$\C$ that contains all the meaningful information of the starting tensor. The matrices $A_1,\dots,A_d$ provide a way to store $\C$ in a bigger space and can be thought as a change of coordinate. Another application of tensors of low multilinear rank is that they serve as building blocks in the so-called \emph{block term decomposition} of a tensor $T\in \CC^{n_1}\otimes \cdots \otimes \CC^{n_d}$ \cite{Yang2014,Lathauwer2011,DL2020,SL2015,Lathauwer2008DecompositionsOA}. This is a decomposition of the form $T=T_1+\cdots+T_r$, where the $T_i$ have bounded multilinear rank. For instance, the paper \cite{HCSPVHL2014} uses block term decompositions for detecting epileptic seizures of patients.


Given a tensor format $n_1\times \cdots \times n_d$, not all multilinear ranks $\bfk$ are attainable. For instance, in the matrix case ($d=2$) the multilinear rank $\bfk=(k_1,k_2)$ with $k_1\neq k_2$ is never realizable, because the column rank of a matrix equals its row rank. Moreover, any rank $\bfk=(1,1,k_3),\, k_3\geq 2,$ is never realizable for a $3$-order tensor~$T\in \CC^{n_1}\otimes \CC^{n_2}\otimes \CC^{n_3}$, since as soon as we impose $\rk(T^{(1)})=\rk(T^{(2)})=1$ then we are forced to also have~$\rk(T^{(3)})=1$ (cf. also \cite{CK11}). 

Once we have fixed a realizable multilinear rank $\bfk=(k_1,\dots,k_d)$, we can consider the space of tensors having multilinear rank bounded by $\bfk$. I.e., $\mlrk(T)\leq \bfk$ means that $\rk(T^{(i)})\leq k_i$ for all $i=1,\ldots,d$.
Further, since multilinear rank is invariant under scalar multiplication, we can pass to the projective space and consider
$$
X_{\bfk}\ :=\ \left\{ T\in \PP^{n_1\cdots n_d-1}\ \mid\ \mlrk(T)\leq \bfk  \right\}.
$$
This variety is the so-called \emph{subspace variety} \cite[Section 3.4.1]{Lands}, which can also be seen as the intersection of some determinantal varieties of matrices with bounded rank.
Landsberg and Weyman \cite{LWtucker} studied the ideal defining $X_{\bfk}$ and proved that $X_{\bfk}$ is normal and has rational singularities. However, even though it is easy to define $X_{\bfk}$, many basic algebraic properties of the subspace variety remain unknown. For instance, \emph{what is the degree of $X_{\bfk}$?} 

The purpose of this paper is to answer the latter question. To state our main result, we denote by~$\Gr(k,n)$ the Grassmannian of $k$-dimensional linear spaces in $\mathbb C^n$. Its dimension is $\dim_{\CC} \Gr(k,n)=k(n-k)$, and its degree (in the Pl\"ucker embedding) is given by the next formula (see, e.g., \cite[Theorem 5.13]{MS}):
$$\deg(\Gr(k,n))\ =\ \frac{(k(n-k))!}{\prod_{j=1}^k j\cdot (j+1)\cdots (j+n-k-1)}.$$
 Furthermore, if we also set $\bfn=(n_1,\dots,n_d)$, where $n_i\geq k_i$ for all $i=1,\dots,d$, then we define a function $f(\bfk, \bfn)$ as follows. Let us write
 $K := k_1\cdots k_d - 1$ and  $D := \sum_{\ell=1}^d \dim_{\CC} \Gr(k_\ell,n_\ell) = \sum_{\ell=1}^d k_\ell(n_\ell-k_\ell).$
 Suppose that $A$ and~$B$ are $k_1\times \cdots \times k_d$ tensors whose entries are the variables $a_{i_1,\dots,i_d}$ and $b_{i_1,\dots,i_d}$ respectively. We collect the entries of $A$ and $B$ in $(K+1)$-vectors $\bfx=(a_{i_1,\dots,i_d})$ and $\bfy=(b_{i_1,\dots,i_d})$. Consider the polynomial $p(\bfx,\bfy)=\prod_{\ell=1}^d \det(A^{(\ell)}(B^{(\ell)})^T)^{n_\ell - k_\ell}$ of bi-degree $(D,D)$ . Set $\alpha=(\alpha_0,\dots,\alpha_K)$ and $\beta=(\beta_0,\dots,\beta_K)$. If~$\sum_{\alpha, \beta} c_{\alpha, \beta} \bfx^\alpha  \bfy^\beta$ is the expansion of $p(\bfx,\bfy) $ in the monomial basis, then 
\begin{equation}\label{def_f}f(\bfk, \bfn)\ :=\ \sum_{\alpha}\, c_{\alpha, \alpha} \cdot \frac{\alpha_0! \cdots \alpha_K!}{(k_1(n_1-k_1))! \cdots (k_d(n_d-k_d))!}.
\end{equation}
The following is then our main result.
\begin{theorem}\label{main}Let $\bfk=(k_1,\dots,k_d)$, $\bfn=(n_1,\dots,n_d)$, where $k_i\leq n_i$ for all $i=1,\dots,d$ and set $N=n_1\cdots n_d-1$.
Suppose that $X_{\bfk}\subset \PP^N$ contains a tensor of multilinear rank equal to $\bfk$. Then,
$$\deg(X_{\bfk})\ =\  \deg(\Gr(k_1,n_1)) \cdots \deg(\Gr(k_d,n_d)) \cdot f(\bfk,\bfn).$$
\end{theorem}
In particular, when $\bfk = \bfn-\mathbf 1$ or $\bfk = \mathbf 1$, where $\mathbf 1$ is the all-one vector, the $\Gr(k_i,n_i)$ are projective spaces and their degree is equal to one. In this case, we therefore get the formulas $\deg(X_{\bfn-\mathbf 1}) =  f(\bfn-\mathbf 1,\bfn)$ and~$\deg(X_{\mathbf 1}) =  f(\mathbf 1,\bfn).$

\begin{example}
In the case $\bfk = \mathbf 1$, the subspace variety becomes the Segre variety of $\PP^{n_1-1}\times \cdots \times \PP^{n_d-1}$ in $\mathbb P^N$. Here, we have $K=0$, $D= \sum_{i=1}^d (n_i-1)$ and $p(\bfx,\bfy)=(x_0y_0)^{D}$, so that $$f(\mathbf 1, \bfn)\ =\  D!\, /\, ((n_1-1)! \cdots (n_d-1)!).$$ Our formula thus yields $\deg(X_{\mathbf 1}) = D! / ((n_1-1)! \cdots (n_d-1)!)$, which is the familiar formula for the degree of the Segre variety (see, e.g., \cite[Example 18.15]{harris}). 
\end{example}
\begin{example}\label{ex2}
The following tables show some degrees of $X_{\bfk} \subset \mathbb P^N$  computed using \Cref{main}.

\begin{minipage}{0.45\textwidth}
\begin{center}
\begin{tabular}{cccc}
\toprule
& \multicolumn{3}{c}{$\bfn$} \\
\cmidrule{2-4}
    $\bfk$ & $(3,3,3)$ &  $(3,3,4)$ & $(3,4,4)$  \\
\midrule 
    $(1,2,2)$ & 108  &   330  &     1560   \\     
    $(2,2,2)$ &  783  & 5175    &  44760  \\     
    $(2,2,3)$ &  306  & 2952    &  19320  \\     
\bottomrule
\end{tabular}
\end{center}
\end{minipage}\quad
\begin{minipage}{0.45\textwidth}
\begin{center}
\begin{tabular}{cccc}
\toprule
& \multicolumn{3}{c}{$\bfn$} \\
\cmidrule{2-4}
    $\bfk$ & $(2,2,3,3)$ &  $(2,3,3,3)$ & $(3,3,3,3)$  \\
\midrule 
    $(1,1,2,2)$ & 216  &   1080  &     5940    \\     
    $(1,2,2,2)$ &  684  & 10962     &   82215  \\     
    $(1,2,2,3)$ &  210  &   4896    &   41616  \\         
\bottomrule
\end{tabular}
\end{center}
\end{minipage}

\end{example}

 To compute the degree of the subspace variety in \Cref{main}, we will use the following result that connects the degree of a projective variety with its volume. The unitary group $\U(N+1)$ acts on $\PP^N$, and there is a unique (up to scaling) Riemannian metric that is invariant under this action; see \mbox{\cite[p.~1216]{cartan1952oeuvres}} or \cite[Section~4.1]{leichtweiss}. Let $\mathrm{vol}(\cdot)$ be the volume measure associated to this metric. The degree of a complex projective variety $X\subset \PP^N$ is then its normalized volume:
	\begin{align}\label{eq:degree_via_volume}
	\deg(X)\ =\ \frac{\vol(X)}{\vol (\PP^{m})}, \quad m=\dim_{\CC} X.
	\end{align}
This is a consequence of the kinematic formula in complex projective space; see \cite[Corollary~3.9]{howard} and \cite[Example 3.12 (b)]{howard}. We also refer to \cite[Section 12.3]{MAG}.

\subsection*{Outline}
In \Cref{section: prelim} we set up our notation and we give preliminaries of both integral and random geometry. \Cref{section: tucker_manifold} is devoted to study the smooth part of the subspace variety as a complex manifold. We realize it as the image of a smooth map and we understand the normal Jacobian of this map. In \Cref{section: volume_tucker} we compute the volume of the subspace variety and prove \Cref{main}. Lastly, \Cref{appendix: code} contains a \texttt{Macaulay2} \cite{M2} code to compute the degrees of subspace varieties.

\subsection*{Acknowledgements}
Both authors have been funded by the Deutsche Forschungsgemeinschaft (DFG) -- Projektnummer 445466444.

\bigskip
\section{Preliminary notions}\label{section: prelim}
Let us start by setting up the notation we will use along the manuscript. Vectors are denoted in boldface. The standard unit vectors are denoted $\bfe_j\in \mathbb R^k$, $j=1,\ldots,k$ and for $i,j=1,\dots,k$, we denote $E_{i,j}:=\bfe_{i}\bfe_j^T$ the $k\times k$ matrix with a $1$ in position $(i,j)$ and $0$ elsewhere.  We denote $\bfn=(n_1,\dots,n_d)$ and $\bfk=(k_1,\dots,k_d) \in \ZZ^{d}_{\geq 0} $, with $k_i\leq n_i$ for all $i=1,\dots,d$. As in the introduction, we set 
\begin{equation}\label{def_NKD}N\ :=\ n_1\cdots n_d-1, \quad  K\ :=\ k_1\cdots k_d-1, \quad \text{ and }\quad D\ :=\ \sum_{\ell=1}^d k_\ell(n_\ell-k_\ell), \end{equation}
so that our space of tensors is $\PP(\mathbb C^{n_1}\otimes \cdots \otimes \mathbb C^{n_d}) \cong \PP^N$, while the core tensor lives in $\PP(\mathbb C^{k_1}\otimes \cdots \otimes \mathbb C^{k_d}) \cong \PP^K$. 

Given a Euclidean vector space $(U,\langle \cdot, \cdot \rangle)$ and two subspaces $V,W\subset U$, we denote $V\perp W$ the case when they are orthogonal; i.e., when for all $v\in V$, $ w\in W$ we have $\langle v,w\rangle = 0$. The orthogonal complement of~$V$ is denoted $V^\perp := \{u\in U \mid \forall v\in V: \langle u,v\rangle = 0\}$.

The projective space  $\PP^N$ is a Riemannian manifold with respect to the \emph{Fubini-Study inner product} $\langle \cdot, \cdot\rangle_T$ defined on the tangent space $\mathrm{T}_T\PP^N$ at a point $T\in \PP^N$. This works as follows. Take a representative $T_0\in\mathbb C^{N+1}$ for $T\in \PP^N$. Then, we have $\mathrm{T}_T\PP^N \cong T_0^\perp = \{Q \in\CC^{N+1} \mid \langle Q,T_0\rangle = 0\}$. The Fubini-Study inner product is given by 
$$\langle Q,R\rangle_T\ =\ \frac{\mathrm{Re}\langle Q,R\rangle}{\Vert T_0\Vert},\quad Q,R\in T_0^\perp,$$
where the latter is the standard Hermitian inner product on $\mathbb C^{N+1}$ and $\Vert T_0\Vert = \langle T_0,T_0\rangle^\frac{1}{2}$ is the norm.

The measure $\vol(\cdot )$ induced by this inner product satisfies the degree--volume formula in \mbox{\Cref{eq:degree_via_volume}}, since it is unitarily invariant. Thus, the goal is to compute the volume of $X_{\bfk}$ with respect to the Fubini-Study metric.
Since it will be used later, let us also recall volume formulae for $\PP^N$:
\begin{equation}\label{vol_PP}
\vol(\PP^N)\ =\ \frac{\pi^N}{N!}.
\end{equation}

\subsection{The Stiefel manifold}
We denote $1_n$ the identity matrix of size $n$ and by $0_{n\times k}$ we denote the~$n\times k$ matrix whose entries are all zeros. The unitary group is denoted by $\U(n)$. It is a real Lie group of dimension $\dim_{\RR} \U(n) = n^2$. The conjugate transpose of $A \in \CC^{n\times k}$ is denoted $A^* \in \CC^{k\times n}$. 

The \emph{Stiefel manifold} is 
$$\St(k,n)\ :=\ \{A \in \CC^{n\times k} \mid A^*A=1_k \}.$$
It is a smooth manifold, and a homogeneous space under $\U(n)$ acting by multiplication from the left. The stabilizer group is $\U(n-k)$, so that $\St(k,n) \cong \U(n) / \U(n-k)$. In particular, the dimension is~$\dim_{\RR} \St(k,n) = \dim_{\RR} \U(n) - \dim_{\RR} \U(n-k ) = 2nk - k^2$. 
We remark that the Stiefel manifold is related to the Grassmannian $\Gr(k,n)$ since $\Gr(k,n)\cong \St(k,n)/\U(k)$. 

In the following lemma we present an orthonormal basis of the tangent space of the Stiefel manifold with respect to the~$\U(n)$-invariant Riemannian metric induced by the inner product 
\begin{equation}\label{IP_St}
    (A,B) \mapsto \frac{1}{2}\,\mathrm{Re}(\mathrm{Trace}(A^*B))
\end{equation}
defined on the ambient space $\mathbb C^{n\times k}$. It is enough to compute the tangent space at a distinguished point.
\begin{lemma}\label{lemma:orthobasis_stiefel}
	Let 
 $A=\left[\begin{smallmatrix}
	1_{k}  \\
	0_{(n-k)\times k} 
\end{smallmatrix}\right] \in \St(k,n)$ 
be the Stiefel matrix with the identity on top and the zero matrix at the bottom. Denote $\iu = \sqrt{-1}$. 
The following $2nk-k^2$ matrices form an orthonormal basis for the tangent space~$\mathrm{T}_{A}\St(k,n)$ of the Stiefel manifold at $A$:
\smallskip
\begin{equation}\label{eq: ortho_basis_stiefel}
\begin{aligned}
&\begin{bmatrix}
E_{i,j} - E_{j,i} \\  0_{(n-k)\times k }
\end{bmatrix}, \; \begin{bmatrix}
	\, \iu\cdot (E_{i,j} +  E_{j,i}) \\  0_{(n-k)\times k }
\end{bmatrix}, 1\leq i< j\leq k, \quad 
\begin{bmatrix}
\sqrt{2}\iu \cdot E_{i,i} \\  0_{(n-k)\times k }
\end{bmatrix}, 1\leq i\leq k,\\[0.5em]
&\begin{bmatrix}
0_{k} \\  \sqrt{2}\cdot E_{i,j}
\end{bmatrix}, \; \begin{bmatrix}
0_{k} \\  \sqrt{2}\iu\cdot E_{i,j}
\end{bmatrix},  1\leq i\leq n-k, 1\leq j\leq k.
\end{aligned}
\end{equation}
\end{lemma}

\begin{proof} 
Consider the smooth map $f:\CC^{n\times k}\to \mathbb C^{k\times k}$ given by $f(M)=M^*M$ and recall that the Stiefel manifold can be realized as $\St(k,n)=f^{-1}(1_k)$. 
The differential of this map is $B \mapsto M^*B+B^*M$ and
knowing how the differential behaves allows us to also compute the tangent space of $\St(k,n)$ at $A$:
\begin{align*}
	\mathrm{T}_{A}\St(k,n)\ &=\ \{ M \in \CC^{n\times k}\mid A^*M=-M^*A   \} \\ &=\ \{ M=\begin{bmatrix}
		B_1 &  B_2
	\end{bmatrix}^T \in \CC^{n\times k}\mid  B_1 \in\mathbb C^{k\times k}: B_1^*=-B_1,\  B_2\in \CC^{(n-k)\times k} \}.
\end{align*}
Thus, the matrices in \Cref{eq: ortho_basis_stiefel} are all in $\mathrm{T}_A\St(k,n)$. Moreover, they are pairwise orthogonal, all of norm one and their number is equal to the dimension of $\St(k,n)$. So they form an orthonormal basis.
\end{proof}

\subsection{The coarea formula}\label{subsection:preliminaries_integral_random}
Let $X, Y$ be Riemannian manifolds with dimensions $\dim(X)\geq \dim(Y)$ and let $F\colon X\rightarrow Y$ be surjective and smooth. Fix a point $\bfx\in X$. The \emph{normal Jacobian} $\mathrm{NJ}(F,\bfx)$ of $F$ at $\bfx$ is 
$$\mathrm{NJ}(F,\bfx)\ =\  \big\vert \det \mathrm D_\bfx F\arrowvert_{(\ker \mathrm D_\bfx F)^\perp} \big\vert,$$
where the matrix representation of the derivative $\mathrm D_\bfx F$ is relative to orthonormal bases in $\mathrm{T}_\bfx X$ and $\mathrm{T}_{F(\bfx)}Y$.
The \emph{coarea formula} (cf. \cite[Section A-2]{howard}) states that for any integrable function $h:X\to \RR$ we have
\begin{equation}\label{coarea_formula}\int_X h(\bfx) \,\mathrm d\bfx\  =\  \int_{\bfy\in Y} \left(\int_{\bfx\in F^{-1}(\bfy)} \frac{h(\bfx)}{\mathrm{NJ}(F,\bfx)} \,\mathrm d \bfx\right)\,\mathrm d \bfy.
\end{equation}

\subsection{Expected values of polynomials evaluated at complex Gaussians}
We denote the standard normal distribution by $N(0,1)$.
The standard complex normal random variable $Z\sim N_{\CC}(0,1)$ on $\CC$ is~$Z=Z_1+\iu \cdot Z_2$, where $Z_1,Z_2\in\RR $ are i.i.d.\ and $Z_1,Z_2\sim N(0,1)$. 

Consider a Gaussian vector ${\bf Z}=(Z_1,\dots,Z_n)\in\mathbb C^n$, where $Z_1,\dots,Z_n \sim N_\CC(0,1)$ are i.i.d. We denote the distribution of ${\bf Z}$ also by $N_{\CC}(0,1)$. 
The norm squared $\Vert {\bf Z}\Vert^2  = \overline Z_1 Z_1 + \cdots +\overline Z_n Z_n $ of ${\bf Z}\sim N_{\CC}(0,1)$ has the \emph{chi-squared distribution} with $2n$ degrees of freedom. We denote this probability distribution by $\chi^2_{2n}$. The moments of $\rho \sim \chi^2_{2n}$ are given by
\begin{equation}\label{eq_chisq}
\mean_{\rho \sim \chi_{2n}^2 }\rho^k\ =\ 2^{k} \cdot\frac{(k+n-1)!}{(n-1)!}.
\end{equation}

The normalization ${\bf Z}/\Vert {\bf Z}\Vert$,  ${\bf Z}\sim N_{\CC}(0,1)$, has the \emph{uniform distribution} on the sphere $\mathbb S^{2n-1}$ and we will denote this distribution by $\mathrm{Unif}(\mathbb{S}^{2n-1})$. 

Let $\alpha=(\alpha_1,\dots,\alpha_n)$, $\beta=(\beta_1,\dots,\beta_n)$ be vectors of non-negative integers and let ${\bf Z}\sim N_\CC(0,1)$. Since it will be useful later, in the following lemma, we understand the expected value of a bi-variate polynomial evaluated at Gaussian vectors.
\begin{lemma}\label{lemma: mean_bivariate_poly}
Consider $f({\bf Z},{\bf{ \overline{Z}}}) :=\sum_{\alpha,\beta}c_{ \alpha, \beta} \cdot  Z_1^{\alpha_1}\overline{Z}_1^{\beta_1}\cdots Z_n^{\alpha_n}\overline{Z}_n^{\beta_n}$. Then,
$$
\mean_{{\bf Z}\sim N_{\CC}(0,1)}f({\bf Z},{\bf{ \overline{Z}}})\ =\ \sum_{ \alpha, \alpha}c_{ \alpha, \alpha} \cdot 2^{\alpha_1+\cdots+\alpha_n}\cdot \alpha_1!\cdots \alpha_n!.
$$
\end{lemma}

\begin{proof}
Since expectation is linear, we have 
$
\mean_{{\bf Z}\sim N_{\CC}(0,1)}f({\bf Z},{\bf{ \overline{Z}}})\  =\  \sum_{ \alpha, \beta}c_{ \alpha,\beta}\cdot \mean \big( Z_1^{\alpha_1}\overline{Z}_1^{\beta_1}\cdots Z_n^{\alpha_n}\overline{Z}_n^{\beta_n}\big),
$
and since $Z_1,\dots,Z_n$ are i.i.d,
\begin{align*}
\mean_{{\bf Z}\sim N_{\CC}(0,1)}f({\bf Z},{\bf{ \overline{Z}}})\ =\ \sum_{\alpha,\beta}c_{ \alpha, \beta}\cdot \mean (Z_1^{\alpha_1}\overline{Z}_1^{\beta_1})\cdots \mean (Z_n^{\alpha_n}\overline{Z}_n^{\beta_n}).
\end{align*}
Hence we can focus on one of the above expected values, say $\mean (Z_1^{\alpha_1}\overline{Z}_1^{\beta_1})$. For every $\phi\in[0,2\pi]$ we have $Z_1\sim e^{\phi\cdot \iu}\cdot Z_1$, so that 
$\mean Z_1^{\alpha_1}\overline{Z_1}^{\beta_1}=e^{\phi\cdot (\alpha_1-\beta_1)\cdot \iu}\mean Z_1^{\alpha_1}\overline{Z}_1^{\beta_1}$. This implies that either $\mean Z_1^{\alpha_1}\overline{Z_1}^{\beta_1} = 0$ or~$\alpha_1=\beta_1$. In the latter case, we have $\mean \left(Z_1\overline{Z_1}\right)^{\alpha_1} = 2^{\alpha_1}\cdot \alpha_1!$ by \Cref{eq_chisq}.
\end{proof}

\bigskip
\section{Metric geometry of the subspace variety}\label{section: tucker_manifold}

In this section, we will prove a few results on the metric geometry of the subspace variety that we will need for the proof of our main result. 
Let $\bfk=(k_1,\dots,k_d)$ be a realizable multilinear rank for $\CC^{n_1}\otimes \cdots \otimes \CC^{n_d}$. That is, we assume that $X_{\bfk}\subset \PP^N$ contains a tensor of multilinear rank equal to $\bfk$.
We denote 
$$Y_{\bfk}\ :=\ \St(k_1,n_1)\times \cdots \times \St(k_d,n_d)\times \PP^K,$$ which can be seen as Riemannian manifold $(Y_{\bfk},g)$, where $g=(h_1,h_2)$ and $h_2$ is the Fubini-Study inner product on $\PP^K$ and $h_1$ is the inner product given by the $d$-uple of $\U(k_i)$-invariant inner products on each~$\St(k_i,n_i)$ from \Cref{IP_St}. 
The subspace variety is
then the image of the smooth map
\begin{align}\label{def_varphi}
	\varphi \colon Y_{\bfk} &\to \PP^N,\quad 
	(A_1,\dots,A_d,\C)\mapsto (A_1\otimes \cdots \otimes A_d)\cdot \C.
\end{align}
Notice that the multilinear rank of $\varphi(A_1,\ldots,A_d,C)$ is equal to $\bfk$, if and only if the multilinear rank of~$C$ is $\bfk$. The smooth part of the subspace variety is
$
X_{\bfk}^\circ=\left\{ T\in \PP^N \mid \mlrk(T)= \bfk  \right\}$. Consequently,
\begin{equation}\label{X_k_smooth}
X_{\bfk}^\circ\ =\ \im(\varphi|_{Y_{\bfk}^\circ}),
\end{equation}
where $Y_{\bfk}^\circ = \St(k_1,n_1)\times \cdots \times \St(k_d,n_d)\times (\PP^K)^\circ$ and $(\PP^K)^\circ  \subset \PP^K$ is the Zariski open set of $k_1\times \cdots\times k_d$ tensors of full multilinear rank.

\begin{lemma}\label{lemma:fiber_tucker}
For $T\in X_{\bfk}^\circ$ where $T=(A_1,\dots,A_d)\cdot \C$ we have
$$
\varphi^{-1}(T)\ =\ \left\{\left(A_1U_1^*,\dots,A_dU_d^*,(U_1\otimes \cdots \otimes U_d)\cdot \C\right)\, \mid\, U_i\in \U(k_i) \hbox{ for all } i=1,\dots,d\right\}.
$$
\end{lemma}
\begin{proof}
For $T\in X_{\bfk}^\circ$ the column rank of the $\ell$-th flattening $T^{(\ell)}$ of $T$ is $k_\ell$ and the column span of $T_\ell$ is the column span of $A_\ell$ for all $\ell=1,\dots,d$. Thus, if $T=(A'_1\otimes \cdots \otimes A'_d)\cdot C'$ then the column span of~$A_\ell$ equals the column span of $A_\ell'$. Hence there exists a unitary matrix $U_\ell\in \U(k_\ell)$ with $A_\ell'=A_\ell U_\ell^*$. Therefore,
\begin{align*}
	T\ =\ (A_1U_1^*,\dots,A_dU_d^*)\cdot C'=(A_1\otimes \cdots \otimes A_d) (U_1^*\otimes \cdots \otimes U_d^*)\cdot C',
\end{align*}
which implies that $C=(A_1^*,\dots,A_d^*)\cdot T=(U_1^*\otimes \cdots \otimes U_d^*)\cdot C'$.
\end{proof}
As immediate consequence, we compute the dimension of the subspace variety.
\begin{corollary}\label{cor_dimension}
The (complex) dimension of the subspace variety $X_{\bfk}\subset \mathbb{P}^N$ is
$$
\dim_{\CC} X_{\bfk}\ =\ D+K,$$
where $D$ and $K$ are defined as in \Cref{def_NKD}.
\end{corollary}
\begin{proof}
The complex dimension of $X_{\bfk}$ is 2 times the dimension of the smooth part $X_{\bfk}^\circ$ as a (real) smooth manifold. Combining \Cref{X_k_smooth} and the fiber dimension theorem, we have   $$
\dim_{\RR} X_{\bfk}^\circ= \sum_{i=1}^d\dim_{\RR} \St(k_i,n_i)+2K-\dim_{\RR} \varphi^{-1}(T)
$$
for $T\in X_{\bfk}^\circ$. By \Cref{lemma:fiber_tucker}, $\dim_{\RR} \varphi^{-1}(T)=\dim_{\RR} \left(\U(k_1)\times \cdots \times \U(k_d)  \right)$. The result follows since we have~$\dim_{\RR} \St(k,n)-\dim_{\RR}\U(k)=2\dim_{\CC} \Gr(k,n)$.
\end{proof}

The group $G_{\bfn}:=\U(n_1)\times \cdots \times \U(n_d)$ acts on $X_{\bfk}$ via
\begin{enumerate}[label={(\alph*)}]
	\item $ (U_1,\dots,U_d).T:=(U_1\otimes \cdots \otimes U_d)\cdot T$,
\end{enumerate}
and it acts on  $Y_{\bfk}$ by
\begin{enumerate}[label={(\alph*)}, resume]
	\item\label{item: action_a} $ (U_1,\dots,U_d).(A_1,\dots,A_d,\C):=(U_1A_1,\dots,U_dA_d,\C)$.
\end{enumerate}
Next, we prove that the map $\varphi$ from \Cref{def_varphi} is equivariant.
\begin{lemma}
The map $\varphi$ is equivariant with respect to the $G_{\bfn}$-actions above.
\end{lemma}
\begin{proof}
We have 
\begin{align*}\varphi((U_1,\dots,U_d).(A_1,\dots,A_d,\C))\ &=\ (U_1A_1\otimes \cdots\otimes U_dA_d)\cdot \C\\
&=\ (U_1\otimes \cdots\otimes U_d)\cdot (A_1\otimes \cdots\otimes A_d)\cdot \C\\
&=\ (U_1\otimes \cdots\otimes U_d)\cdot \varphi(A_1,\ldots, A_d, \C). \qedhere
\end{align*}
\end{proof}

Since $(U_1,\ldots,U_d)\mapsto U_1\otimes \cdots \otimes U_d$ is a representation of $\U(n_1)\times \cdots \times \U(n_d)$ into $\U(N+1)$, and since all involved inner products are unitarily invariant, the above actions are isometric. 
There is another action on $Y_{\bfk}$ that is given by
$G_{\bfk}:=\U(k_1)\times \cdots \times \U(k_d)$ as:
\begin{enumerate}[resume, label={(\alph*)}]
\item $(U_1,\dots,U_d).(A_1,\dots,A_d,\C)\ :=\ (A_1U_1,\dots,A_dU_d,(U_1^*\otimes \cdots \otimes U_d^*)\cdot \C ) .$
\end{enumerate}
Similarly, this action is also isometric. 
We summarize this in the next lemma. 
\begin{lemma}\label{lemma:Gn_action_surj_iso}
The $G_{\bfn}$-actions on both $X_{\bfk}$ and $Y_{\bfk}$ are isometric. The $G_{\bfk}$-action on $Y_{\bfk}$ is isometric.
\end{lemma}

\begin{corollary}\label{cor:volume_fiber}
For $T\in X_{\bfk}^\circ$ we have $\vol(\varphi^{-1}(T))=\vol \U(k_1)\cdots \vol \U(k_d)$. 
\end{corollary}
\begin{proof}
The fiber $\varphi^{-1}(T)$ of $T\in X_{\bfk}^\circ$ is a $G_{\bfk}$-orbit by \Cref{lemma:fiber_tucker}. The action of $\U(k_1)\times \cdots \times \U(k_d)$ is isometric by \Cref{lemma:Gn_action_surj_iso} and the stabilizer of a point is trivial. This implies that the volume of the orbit is the volume of $\U(k_1)\times \cdots \times \U(k_d)$.
\end{proof}

\subsection{The normal Jacobian of $\varphi$} In this part, we compute the normal Jacobian of the map $\varphi$. To simplify the computation we will exploit the group actions just introduced. 

Recall the definition of the Euclidean dense smooth submanifold $Y_{\bfk}^\circ \subset Y_{\bfk}$. 
Since the group action~\ref{item: action_a} above is isometric (\Cref{lemma:Gn_action_surj_iso}), the normal Jacobian at $(A_1,\ldots, A_d, C)\in Y_{\bfk}^\circ$ does not depend on $A_1,\ldots, A_d$. We summarize this in a lemma.
\begin{lemma}\label{NJ_ind}
The normal Jacobian $\mathrm{NJ}(\varphi, (A_1,\ldots,A_d,C))$ of $\varphi$ at $(A_1,\ldots,A_d,C)$ only depends on $C$, but not on $A_1,\ldots,A_d$.
\end{lemma}
Thus, we can assume in the following that for all $i=1,\dots,d$
$$A_i\ :=\ \begin{bmatrix}
	1_{k_i}  \\
	0_{(n_i-k_i)\times k_i} 
\end{bmatrix}.$$
The differential of the map $\varphi$ at $(A_1,\ldots, A_d,C)$ is given by 
\begin{equation}\label{diff_varphi_eval}\mathrm D_{(A_1,\ldots, A_d,C)}\varphi(\dot A_1,\ldots, \dot A_d,\dot C)\ =\  A_1\otimes \cdots \otimes A_d\cdot \dot{C} + \sum_{i=1}^d(A_1\otimes \cdots \otimes \dot{A}_i\otimes \cdots \otimes A_d)\cdot C.
\end{equation}
To compute the normal Jacobian at a point $(A_1,\ldots, A_d,C)\in Y_{\bfk}^\circ$ we must compute an orthonormal basis for $(\ker \mathrm D_{(A_1,\ldots, A_d,C)} \varphi)^\perp\subset \mathrm{T}_{(A_1,\ldots,A_d,C)}  Y_{\bfk}$ and evaluate the differential in that basis.

Let us first compute an orthonormal basis for $(\ker \mathrm D_{(A_1,\ldots, A_d,C)} \varphi)^\perp$. Write $T:=\varphi(A_1,\ldots,A_d,C)\in Y_{\bfk}^\circ$. We have that 
\begin{equation}\label{ker_phi}\ker \mathrm D_{(A_1,\ldots, A_d,C)}\ =\ \mathrm{T}_{(A_1,\ldots, A_d,C)}\varphi^{-1}(T).\end{equation}
By 
\Cref{lemma:fiber_tucker}, 
for $T\in X_{\bfk}^\circ$, where $T=(A_1,\dots,A_d)\cdot \C$, we have
\begin{equation}\label{ker_phi2}
\varphi^{-1}(T)\ =\ \left\{\left(\begin{bmatrix}
	U_1^* \\  0_{k_1\times (n_1-k_1)}
\end{bmatrix},\dots,\begin{bmatrix}
	U_d^* \\  0_{k_d\times (n_d-k_d)}
\end{bmatrix},(U_1\otimes \cdots \otimes U_d)\cdot \C\right) \ \Bigm|\  U_i\in \U(k_i)\right\}.
\end{equation}

Let us denote by $\mathcal B_\ell$ the set composed of the following orthonormal  vectors in $\mathrm{T}_{(A_1,\ldots,A_d,C)}  Y_{\bfk}$: 
\begin{equation}\label{def_B_l}
\left(0_{n_1\times k_1},\dots,\begin{bmatrix}
	0_{k_\ell} \\  E_{i,j}
\end{bmatrix},\dots,0_{n_d\times k_d},0_{k_1\times \cdots \times k_d}\right),\; \left(0_{n_1\times k_1},\dots,
\begin{bmatrix}
	0_{k_\ell} \\  \iu \cdot E_{i,j}
\end{bmatrix},\dots,0_{n_d\times k_d},0_{k_1\times \cdots \times k_d}\right),
\end{equation}
for $1\leq i\leq k_\ell$ and $1\leq j\leq n_\ell-k_\ell$, where the non-zero element is at the $\ell$-th position and we denoted $0_{k_1\times \cdots\times k_d}$ the zero tensor in $\CC^{n_1}\otimes \cdots \otimes \CC^{k_d}$. The union $\mathcal B:=\mathcal B_1\cup\cdots\cup\mathcal B_d$ is composed of orthonormal vectors. Let us denote their span by
$$V\ :=\ \mathrm{span}(\mathcal B).$$
Then, it follows from \Cref{ker_phi} and \Cref{ker_phi2} that 
$V \subset (\ker \mathrm D_{(A_1,\ldots, A_d,C)} \varphi)^\perp$.
Next, we denote by~$\mathcal A$ and orthonormal basis for 
$V^\perp \cap (\ker \mathrm D_{(A_1,\ldots, A_d,C)} \varphi)^\perp.$ Consequently, 
$$\mathcal A\cup \mathcal B \text{ is an orthonormal basis for } (\ker \mathrm D_{(A_1,\ldots, A_d,C)} \varphi)^\perp.$$ 

Now that we understood a basis of $(\ker \mathrm D_{(A_1,\ldots, A_d,C)} \varphi)^\perp$, we want to evaluate the differential in this basis. The next lemma shows that the derivative of $\varphi$ maps $\mathcal A,\mathcal B_1,\ldots,\mathcal B_d$ to pairwise orthogonal spaces.

\begin{lemma}\label{lemma:matrices_W_and_Q}
Let $\bfp=(A_1,\dots,A_d,\C)$ and $T=\varphi(\bfp)$.	We have the following.
\begin{enumerate}
\item\label{item:1_lemma_matrices} If $\bfu\in \mathcal A ,\bfv\in \mathcal B_\ell$ then $ \langle \mathrm D_\bfp\varphi(\bfu), \mathrm D_\bfp\varphi(\bfv)\rangle_T=0$.  
\item\label{item:2_lemma_matrices}  
If $\bfv\in \B_\ell,\bfw\in\B_{\ell'}$ and $\ell\neq \ell'$, then $ \langle \mathrm D_\bfp\varphi(\bfv), \mathrm D_\bfp\varphi(\bfw)\rangle_T=0$ .
\end{enumerate}
\end{lemma}
\begin{proof}
Pick a representative $C_0$ of $C$ of norm 1. Then, $T_0 = (A_1\otimes \cdots \otimes A_d)\cdot C_0$ is a representative of~$T=\varphi(\bfp)$ and we have 
$\langle T_0,T_0\rangle\ =\ \langle (A_1^*A_1\otimes \cdots \otimes A_d^*A_d)\cdot C_0, C_0\rangle\ =\ \langle C_0,C_0\rangle\ =\ 1.$
Thus, on the tangent space $\mathrm{T}_T X_{\bfk}\cong T_0^\perp$ we have the inner product $\langle \bfa, \bfb\rangle_T = \mathrm{Re} \langle \bfa, \bfb\rangle$.

It follows from \Cref{eq: ortho_basis_stiefel} that tangent vectors in $V^\perp \subset \mathrm{T}_\bfp Y_{\bfk}$, and thus in $\mathcal A$, are of the form
$$\bfu = \left(\begin{bmatrix}
	\dot Q_1 \\  0_{k_1\times (n_1-k_1)} 
\end{bmatrix}, \dots, \begin{bmatrix}
	\dot Q_d \\  0_{k_1\times (n_d-k_d)} 
\end{bmatrix}, \dot C\right) = \left(A_1\dot Q_1, \dots, A_d\dot Q_d, \dot C\right)\in\mathcal A,$$
where $\dot Q_i\in \mathbb C^{k_i\times k_i}$ and $\dot C\in \mathbb C^{K+1}$. Let us also pick $\bfv \in\mathcal B_\ell$ of the form 
$$\bfv\ =\ \left(0_{n_1\times k_1},\dots,\begin{bmatrix}
	0_{k_\ell} \\  E_{i,j}
\end{bmatrix},\dots,0_{n_d\times k_d},0_{k_1\times \cdots \times k_d}\right)\in\mathcal B_\ell.$$

The differential $\mathrm D_\bfp\varphi$ evaluated at these vectors gives, following \Cref{diff_varphi_eval}, 
\begin{align*}
\bfa\ :=\ \mathrm D_\bfp\varphi(\bfu)\ &=\ A_1\otimes \cdots\otimes A_d\cdot \Big( \dot{C} + \sum_{i=1}^d(1_{k_1}\otimes \cdots \otimes \dot{Q}_i\otimes \cdots \otimes 1_{k_d})\cdot C\Big)\ \in\ T_0^\perp,\\
\bfb\ :=\ \mathrm D_\bfp\varphi(\bfv)\ &=\ \left(A_1\otimes \cdots \otimes \begin{bmatrix}
	0_{k_{\ell}} \\  E_{i,j}
\end{bmatrix}\otimes \cdots \otimes A_d\right)\cdot C\ \in\ T_0^\perp.
\end{align*}
Consequently,
$\langle \bfa,\bfb\rangle_T = \mathrm{Re}\mathrm\langle\mathrm \bfa,\bfb\rangle = 0,$ 
since $\bfa$ is in the image of $A_1\otimes \cdots \otimes A_d$ and 
\begin{align*}
(A_1^*\otimes \cdots \otimes A_d^*) \cdot \bfb\ &=\left(1_{k_1}\otimes \cdots \otimes A_\ell^*\begin{bmatrix}
	0_{k_{\ell}} \\  E_{i,j}
\end{bmatrix}\otimes \cdots \otimes 1_{k_d}\right)\cdot C\\
&=\ \left(1_{k_1}\otimes \cdots \otimes 0_{n_\ell\times k_\ell}\otimes \cdots \otimes 1_{k_d}\right)\cdot C\ =\ 0.
\end{align*}
By multilinearity, also $\langle \bfa,\bfb'\rangle_T=0$, where $\bfb' =\ \mathrm D_\bfp\varphi(\iu\cdot \bfv).$
This proves the first item. 

Next, we choose 
$$\bfw\ =\ \left(0_{n_1\times k_1},\dots,\begin{bmatrix}
	0_{k_{\ell'}} \\  E_{r,s}
\end{bmatrix},\dots,0_{n_d\times k_d},0_{k_1\times \cdots \times k_d}\right)\in\mathcal B_{\ell'}$$
and we denote the image of $\bfw$ by  $\bfc := \mathrm D_\bfp\varphi(\bfw)\in T_0^\perp$.
Then, we have 
\begin{equation}\label{bc}
\begin{aligned}
\langle \bfb, \bfc\rangle\ &=\ \left\langle \left(A_1\otimes \cdots \otimes \begin{bmatrix}
	0_{k_{\ell}} \\  E_{i,j}
\end{bmatrix}\otimes \cdots \otimes A_d\right)\cdot C, \left(A_1\otimes \cdots \otimes \begin{bmatrix}
	0_{k_{\ell'}} \\  E_{r,s}
\end{bmatrix}\otimes \cdots \otimes A_d\right)\cdot C\right\rangle\\[0.5em]
&=\ \left\langle \C, \left(1_{k_1}\otimes \cdots \otimes \begin{bmatrix}
	0_{k_{\ell}} \\  E_{i,j}
\end{bmatrix}A_{\ell}\otimes \cdots \otimes A_{\ell'}^* \begin{bmatrix}
0_{k_{\ell'}} \\  E_{r,s}
\end{bmatrix}\otimes \cdots \otimes 1_{k_d}\right)\cdot \C   \right\rangle .	
\end{aligned}
\end{equation}
If $\ell\neq \ell'$ then $\begin{bmatrix}
	0_{k_\ell} &  E_{i,j}^T
\end{bmatrix}^TA_{\ell}= 0_{k_\ell}$, hence $\langle \bfb, \bfc\rangle_T=\mathrm{Re}\langle \bfb, \bfc\rangle = 0$. By multilinearity also $\langle \bfb', \bfc\rangle_T=0$, where as before $\bfb' =\ \mathrm D_\bfp\varphi(\iu\cdot \bfv).$  This shows the second item. 
\end{proof}

We can now compute the normal Jacobian of $\varphi$ at $(A_1,\ldots,A_d,C)$. Notice that $\mathcal A$ has $2K$ elements.
\begin{theorem}\label{thm_NJ} 
Let $\bfp=(A_1,\dots,A_d,\C)$ and $T=\varphi(\bfp)$. Let $\C^{(\ell)}$ be the $\ell$-th flattening of a representative of~$\C$ of norm one. Recall the definition of the orthonormal basis $\mathcal A$ above. Suppose that $\mathcal A= \{\dot C_1,\ldots, \dot C_{2K}\}$. The normal Jacobian of $\varphi$ at $\bfp$ is 
$$\NJ(\varphi,\bfp)\ =\ \sqrt{\det(R)}\cdot \prod_{\ell=1}^d \det \left( \C^{(\ell)}(C^{(\ell)})^*\right)^{n_\ell-k_\ell},$$
where $R$ is the Gram matrix of $\mathcal A$; i.e.,
$$R=\left[\left\langle \mathrm D_\bfp(\varphi)(\bfv), \mathrm D_\bfp(\varphi)(\bfw)\right\rangle_T \right]_{\bfv,\bfw\in\mathcal A} \in\mathbb R^{2K\times 2K}.$$
\end{theorem}
\begin{proof}
As we recalled in \Cref{subsection:preliminaries_integral_random}, the normal Jacobian $\NJ(\varphi,\bfp) $ is the determinant of the differential matrix restricted to $ (\ker \mathrm D_\bfp\varphi)^\perp$ represented in coordinates by orthonormal bases. 

Recall from above the definition of the orthogonal basis $\mathcal A\cup \mathcal B_1\cup\cdots \cup \mathcal B_d$ of $ (\ker \mathrm D_\bfp\varphi)^\perp$. 
We denote by~$W_\ell$ the Gram matrix of $\mathcal B_\ell$: 
$$
	W_{\ell}\ =\  \left[ \left\langle \mathrm D_\bfp(\varphi)(\bfv), \mathrm D_\bfp(\varphi)(\bfw)\right\rangle_T \right]_{\bfv,\bfw\in\mathcal B_\ell}\in \mathbb R^{2k_\ell(n_\ell-k_\ell)\times 2k_\ell(n_\ell-k_\ell)}.$$
Then, by \Cref{lemma:matrices_W_and_Q}, the square of the normal Jacobian would be represented as the determinant of the following block diagonal matrix \begin{align*}\label{eq:normal_jacobian}
	\NJ(\varphi,p)^2\ =\ \det	
 \begin{bmatrix}
W_{1} &  &  & \\
 &      \ddots    &      &  \\
& & W_{d} &  \\
&  & & R
	\end{bmatrix}.
\end{align*}
Thus, 
$
	\NJ(\varphi,p)^2= \det R\cdot \prod_{\ell=1}^d \det W_{\ell}.
$
To conclude now, it is enough to compute the determinant of $W_\ell$. 

As in the proof of \Cref{lemma:matrices_W_and_Q} we choose a representative $C_0$ of $C$ of norm 1, so that $T_0 = (A_1\otimes \cdots \otimes A_d)\cdot C_0$ is a representative of~$T$ of norm one, and on $\mathrm{T}_{T_0} X_{\bfk}\cong T_0^\perp$ we have the inner product $\langle \bfa, \bfb\rangle_T = \mathrm{Re}\langle \bfa,\bfb\rangle$.

Recall from \Cref{def_B_l} that $\mathcal B_\ell$ is composed of vectors $\bfv_{i,j} := \left(0_{n_1\times k_1},\dots,\left[\begin{smallmatrix}
	0_{k_\ell} \\  E_{i,j}
\end{smallmatrix}\right],\dots,0_{n_d\times k_d},0_{k_1\times \cdots \times k_d}\right)$ and $\iu\cdot \bfv_{i,j}$. 
In particular, the span of $\mathcal B_\ell$ is a complex linear space. From \Cref{diff_varphi_eval} we see that $\mathrm D_{\bfp}\varphi$ is a complex linear map. This implies that
$$\det W_\ell = \left\vert\det \begin{pmatrix}\langle \bfb_{i,j}, \bfb_{r,s}\rangle\end{pmatrix}_{(i,j),(r,s)}\right\vert^2,\quad \text{ where } \bfb_{i,j}\ :=\ \mathrm{D}_{\bfp}\varphi(\bfv_{i,j})\ \in\ T_0^\perp.$$
From \Cref{bc} we get
\begin{equation*}
\begin{aligned}
\langle \bfb_{i,j}, \bfb_{r,s}\rangle\ &=\ \left\langle \left(A_1\otimes \cdots \otimes \begin{bmatrix}
	0_{k_{\ell}} \\  E_{i,j}
\end{bmatrix}\otimes \cdots \otimes A_d\right)\cdot C, \left(A_1\otimes \cdots \otimes \begin{bmatrix}
	0_{k_{\ell}} \\  E_{r,s}
\end{bmatrix}\otimes \cdots \otimes A_d\right)\cdot C\right\rangle\\[0.5em]
&=\ \left\langle \C, \left(1_{k_1}\otimes \cdots \otimes \begin{bmatrix}
0_{k_i}\\ E_{i,j}^T\, E_{r,s}
\end{bmatrix}  \otimes \cdots \otimes 1_{k_d}\right)\cdot \C   \right\rangle.
\end{aligned}
\end{equation*}
Since $E_{i,j}^TE_{r,s}=\bfe_j\bfe_i^T\bfe_r\bfe_s^T$ is non-zero if and only if $i=r$, we have $
\langle \bfb_{i,j}, \bfb_{r,s}\rangle =0$ for~$i\neq r$. 

We denote by $\C^{(\ell)}_j$ the $j$-th slice of the $\ell$-th flattening of $\C$. In the case~$i=r$ we have 
\begin{align*}
    \langle \bfb_{i,j}, \bfb_{r,s}\rangle\ &=\ \langle \C, (1_{k_1}\otimes \cdots \otimes E_{j,s}\otimes \cdots \otimes 1_{k_d})\cdot \C   \rangle\\
   &=\ \langle (1_{k_1}\otimes \cdots \otimes \bfe_{j}^T\otimes \cdots \otimes 1_{k_d})\cdot \C, (1_{k_1}\otimes \cdots \otimes \bfe_{s}^T\otimes \cdots \otimes 1_{k_d})\cdot \C   \rangle\\
&=\ \langle \C^{(\ell)}_j,  \C^{(\ell)}_s\rangle.
\end{align*}
 Thus, the matrix $ \begin{pmatrix}\langle \bfb_{i,j}, \bfb_{r,s}\rangle\end{pmatrix}_{(i,j),(r,s)}$ has a block diagonal structure with $n_\ell -k_\ell$ blocks of size $k_\ell\times k_\ell$:
$$ \begin{pmatrix}\langle \bfb_{i,j}, \bfb_{r,s}\rangle\end{pmatrix}_{(i,j),(r,s)}\  =\  \begin{bmatrix} 
C^{(\ell)}(C^{(\ell)})^*  && \\
  &\ddots & \\
  && C^{(\ell)}(C^{(\ell)})^*
\end{bmatrix}\in \CC^{k_\ell(n_\ell-k_\ell)\times k_\ell(n_\ell-k_\ell)}.
$$
Thus, $\det W_\ell = \vert\det( \begin{pmatrix}\langle \bfb_{i,j}, \bfb_{r,s}\rangle\end{pmatrix})\vert^2 = \det ( C^{(\ell)}(C^{(\ell)})^*)^{2n_\ell(n_\ell - k_\ell)}$, 
which concludes the proof.
\end{proof}

\bigskip
\section{The volume of the subspace variety}\label{section: volume_tucker}
This section is entirely to compute the volume of the subspace variety $X_{\bfk}$. At the end, we use this to prove \Cref{main}.

As before, we assume that $X_{\bfk}\subset \PP^N$ contains a tensor of multilinear rank equal to $\bfk$. Let $X_{\bfk}^\circ$ be the smooth part. We consider a measurable function $h:X_{\bfk} \to\mathbb R$. 
Recall from \Cref{section: tucker_manifold} the definition of $\varphi$ and $Y_{\bfk}^\circ$. In particular, $X_{\bfk}^\circ$ and $Y_{\bfk}^\circ$ are smooth manifolds and $\varphi : Y_{\bfk}^\circ\to X_{\bfk}^\circ$ is surjective. We can therefore apply the coarea formula from \Cref{coarea_formula} to get 
\begin{equation}\label{eq4}
\int_{Y_{\bfk}^\circ}\NJ(\varphi,\bfp)\cdot (h\circ\varphi)(\bfp)\; \mathrm d \bfp\ =\ \int_{X_{\bfk}^\circ} \mathrm{vol}(\varphi^{-1}(T)) \cdot h(T)\; \mathrm d T \ =\ \vol \U(k_1)\cdots \vol \U(k_d)\cdot \int_{X_{\bfk}^\circ }h(T)\; \mathrm d T,
\end{equation}
where for the last equation we used 
that for $T\in X_{\bfk}^\circ$ we have $\vol(\varphi^{-1}(T))=\vol \U(k_1)\cdots \vol \U(k_d)$ (see \Cref{cor:volume_fiber}). If we take for $h$ the function that is constant and equal to one, and using that $\vol(X_{\bfk}) = \vol(X_{\bfk}^\circ)$, we get from \Cref{eq4} the following formula for the volume of the subspace variety:
$$\vol(X_{\bfk})\ =\ \int_{X_{\bfk}^\circ} 1\; \mathrm d T\ =\ \frac{1}{\vol \U(k_1)\cdots \vol \U(k_d)}\cdot \int_{Y_{\bfk}^\circ}\NJ(\varphi,\bfp)\; \mathrm d \bfp.$$
The Grassmannian can be seen as the quotient space $\Gr(k_i,n_i) \cong \St(k_i,n_i) / \U(k_i)$. On the other hand, it is a projective variety in the Pl\"ucker embedding and thus has a volume. Consequently, the volume therefore is given by $\mathrm{vol}(\Gr(k_i,n_i))=\mathrm{vol}(\St(k_i,n_i)) / \mathrm{vol}(\U(k_i))$ for all $i=1,\dots,d$. Using this formula for the volume of the Grassmannian, \Cref{NJ_ind} and \Cref{thm_NJ} we get that 
$$\vol(X_{\bfk})\ =\ \bigg(\prod_{i=1}^d\vol \Gr(k_i,n_i)\bigg) \cdot  \int_{\PP^K} \sqrt{\det(R)}\cdot \prod_{\ell=1}^d \det \left( \C^{(\ell)}(C^{(\ell)})^*\right)^{n_\ell-k_\ell} \; \mathrm d C.$$

The following lemma allows us to better understand the above integral. The idea behind this result comes from the paper \cite{Beltr2009} by Beltr\'an, where he computes the volume of the variety of (real) matrices of bounded rank.
\begin{lemma}\label{lemma:R_disappear}
We have
$$
\int_{\PP^K}\sqrt{\det R} \cdot \prod_{\ell=1}^d  \det \left(\C^{(\ell)}(\C^{(\ell) })^* \right)^{n_\ell-k_\ell} \; \mathrm d C\ =\ 
\int_{\PP^K} \prod_{\ell=1}^d\det\left(C^{(\ell)} (C^{(\ell)})^* \right)^{n_\ell-k_\ell}\; \mathrm d C.
$$
\end{lemma}
\begin{proof}
We consider the case where $n_i=k_i$ for all $i=1,\dots,d$. In this case, the subspace variety is $X_{\bfk}=\PP^K$, we have $Y_{\bfk}=\U(k_1)\times \cdots \times \U(k_d)\times \PP^K$, and the matrix $R$ yields the normal Jacobian of $\varphi$. 
We treat the factor $h(C): = \prod_{\ell=1}^d\det(\C^{(\ell)}(\C^{(\ell) })^* )^{n_\ell-k_\ell}$ as a separate integrable function. Observe that $h(C)$ is constant on fibers of $\varphi$, so that $h\circ \varphi = \prod_{\ell=1}^d\det(\C^{(\ell)}(\C^{(\ell) })^* )^{n_\ell-k_\ell}$. \Cref{eq4} implies
\begin{align*}
\int_{\PP^K}\sqrt{\det R} \cdot h(C)\;\mathrm d C\ &= \ \frac{1}{\vol \U(k_1)\cdots \vol \U(k_d)}\cdot \int_{Y_{\bfk}^\circ} \sqrt{\det R} \cdot  h(C)\;\mathrm d C\
=\  \int_{\PP^K}  h(C)\;\mathrm d C.\qedhere
\end{align*}
\end{proof}

Let us get back to the last expression of $\vol(X_{\bfk})$. Using \Cref{lemma:R_disappear} we obtain
\begin{align*}
\vol(X_{\bfk})&=\ \bigg(\prod_{i=1}^d\vol \Gr(k_i,n_i)\bigg) \cdot \int_{\PP^K}\prod_{\ell=1}^d\det\left(T^{(\ell)}(T^{(\ell) *}) \right)^{n_\ell-k_\ell} \mathrm d T.
\end{align*}
To conclude our discussion, it is enough to understand how to compute the above integral. We start by setting
\begin{equation}\label{eq5}
\Lambda (\bfn,\bfk)\ :=\ \displaystyle\mean_{T \sim \Unif( \PP^K)}\; \prod_{\ell=1}^d \det\left( T^{(\ell)}(T^{(\ell)*}) \right)^{n_\ell-k_\ell}.
\end{equation}
Thus,
$$\vol(X_{\bfk})\ =\ \bigg(\prod_{i=1}^d\vol \Gr(k_i,n_i)\bigg) \cdot \vol(\PP^K) \cdot \Lambda (\bfn,\bfk).$$
The remaining task is now to evaluate  $\Lambda (\bfn,\bfk)$.

Recall from \Cref{def_NKD} that we denote $K=k_1\cdots  k_d-1$ and $D = \sum_{\ell=1}^d k_\ell(n_\ell-k_\ell).$
In \Cref{eq5} we can also take the expectation over the uniform distribution on the sphere $\mathbb{S}^{2K+1}$. I.e.,
$$\Lambda (\bfn,\bfk)\ =\ \displaystyle\mean_{T \sim \Unif( \mathbb{S}^{2K+1})}\; \prod_{\ell=1}^d \det\left( T^{(\ell)}T^{(\ell)*} \right)^{n_\ell-k_\ell}.$$
Suppose that $T\sim \Unif(\mathbb S^{2K+1})$ and $\rho \sim \chi^2_{2(K+1)}$  are independent. If $q:=\sqrt{\rho}$ then~$q\cdot T\in\mathbb C^{K+1}$ is~$N_{\CC}(0,1)$. Thus,
\begin{align*}
\Lambda (\bfn,\bfk)\ 
&=\ \left( \mean_{ \rho\sim \chi^2_{2(K+1)}} \rho^{D} \right)^{-1}\cdot \mean_{ T \sim \Unif( \mathbb{S}^{2K+1})\atop \rho\sim \chi^2_{2(K+1)}} \;
\prod_{\ell=1}^d \det \left( (q\cdot T)^{(\ell)}(q\cdot T)^{(\ell)*} \right)^{n_\ell-k_\ell}\qquad (\text{where } q=\sqrt{\rho})\\[0.3em]
&=\  \frac{K!}{(D+K)!\, 2^D}\,\cdot \mean_{T\sim N_{\CC}(0,1)}\; \prod_{\ell=1}^d \det \left( T^{(\ell)} (T^{(\ell)*}) \right)^{n_\ell-k_\ell},
\end{align*}
where for the last equality we have also used from \Cref{eq_chisq} that
$
\mean_{\rho\sim \chi^2_{2(K+1)}}  \rho^{D}\ =\ 2^{ D}\cdot  \frac{(D+K)!}{K!}.
$

Now it only remains to understand the expected value appearing in $\Lambda(\bfn,\bfk)$. For this, let us recall the function $f(\bfk,\bfn)$ defined in \Cref{def_f}. Denote by $\bfx$ the $K+1$ coordinates in $\CC^{K+1}$ of a general tensor $T\in \CC^{k_1}\otimes \cdots \otimes \CC^{k_d}$ and by $\bfy$ the coordinates of $T^*$. We denote $c_{\alpha,\beta}$ the coefficients of the monomial expansion of the bi-variate polynomial $p(\bfx,\bfy)=\prod_{\ell=1}^d \det \left( T^{(\ell)} (T^{(\ell)*}) \right)^{n_\ell-k_\ell}$ of bi-degree $(D,D)$. Then,
$$
f(\bfk, \bfn)\ =\ \sum_{\alpha}\, c_{\alpha, \alpha} \cdot \frac{\alpha_0! \cdots \alpha_K!}{(k_1(n_1-k_1))! \cdots (k_d(n_d-k_d))!}.
$$
Applying \Cref{lemma: mean_bivariate_poly} to $\mean_{T\sim N_{\CC}(0,1)} p(\bfx,\bfy)$, gives the following identity. 
\begin{lemma}We have
$\displaystyle \Big(\prod_{i=1}^d(k_i(n_i-k_i))!\Big)\cdot f(\bfk, \bfn) = \frac{1}{2^{ D}} \cdot \mean_{T\sim N_{\CC}(0,1)}\prod_{\ell=1}^d \det \left( T^{(\ell)} (T^{(\ell)*}) \right)^{n_\ell-k_\ell}.$
\end{lemma}

Consequently,
\begin{equation}\label{eq1}
    \vol(X_{\bfk})\ =\ \frac{\prod_{i=1}^d(k_i(n_i-k_i))!}{\pi^{D}}\cdot \Big(\prod_{i=1}^d\vol(\Gr(k_i,n_i))\Big) \cdot \vol(\PP^{D +K}) \cdot f(\bfk,\bfn),
\end{equation}
where we have used the formula in \Cref{vol_PP} for the volume of projective space.

We have now all ingredients to prove our main \Cref{main}. 
\begin{proof}[Proof of \Cref{main}]
From the degree formula via volumes of \Cref{eq:degree_via_volume}, and using from \Cref{cor_dimension} that the dimension of the subspace variety $X_{\bfk}$ is
$
D+K
$, we get
$$
\deg(X_{\bfk})\ =\ 
\frac{\vol(X_{\bfk})}{\vol(\PP^{D+K})} .
$$
Thus, \Cref{eq1} implies that
$$
\deg(X_{\bfk})\ =\  \frac{\prod_{i=1}^d(k_i(n_i-k_i))!}{\pi^{D}}\cdot \Big(\prod_{i=1}^d\vol(\Gr(k_i,n_i))\Big)  \cdot f(\bfk,\bfn).
$$
Using from \Cref{vol_PP} that $\vol \Gr(k,n) = \deg(\Gr(k,n)) \cdot \pi^{k(n-k)} / (k(n-k))!$ we get 
$$
\deg(X_{\bfk})\ =\ \deg \Gr(k_1,n_1) \cdots \deg \Gr(k_d, n_d)  \cdot f(\bfk,\bfn),\,
$$
and this proves \Cref{main}.
\end{proof}

\bigskip

\bibliographystyle{alpha}
\bibliography{References.bib}

\bigskip

\newpage
\appendix
\bigskip
\section{Macaulay2 code computing $\deg(X_{\bfk})$}\label{appendix: code}
In this appendix, we include a \texttt{Macaulay2} \cite{M2} code to compute the degree of a subspace variety $X_{\bfk}\subset \PP^{N}$ with the formula of \Cref{main},
which relies on the degrees of all $\Gr(k_i,n_i)$ and on $f(\bfk,\bfn)$. The code runs on \texttt{Macaulay2 1.22}.

Let us start by creating a function that computes the degree of a Grassmannian variety $\Gr(q,r)\subset \PP^{{r \choose q}-1}$.

{\small
\begin{lstlisting}[language=code]
degG = (q,r)->(      
    num := (q*(r-q))!;
    den := product apply(toList(1..q),x->product toList(x..(x+r-q-1)));
    return num/den
    )
\end{lstlisting}
}

Now we need to take care of $f(\bfk,\bfn)$. Call $A,B\in \CC^{k_1}\otimes \cdots \otimes \CC^{k_d}$ two general tensors whose entries are the variables $\bfx:=(a_{i_1,\dots,i_d})$ and $\bfy:=(b_{i_1,\dots,i_d})$ respectively. If $\sum_{\alpha, \beta} c_{\alpha, \beta} \bfx^\alpha  \bfy^\beta$ is the expansion of $\prod_{\ell=1}^d \det(A^{(\ell)}(B^{(\ell)})^T)^{n_\ell - k_\ell}$ in the monomial basis, then we denote 
$$g(\bfk,\bfn)\ :=\ \sum_{\alpha}\, c_{\alpha, \alpha} \cdot \alpha_0! \cdots \alpha_K!.$$
Thus, $g(\bfk, \bfn) = (k_1(n_1-k_1))! \cdots (k_d(n_d-k_d))!\cdot f(\bfk, \bfn)$.
In other words, $g(\bfk,\bfn)$ contains the coefficients of the terms of $ \prod_{\ell=1}^d\det(A^{(\ell)}(B^{(\ell)})^T)^{n_\ell-k_\ell}$ that have the same bi-degree $(D,D)$, where $D=\sum k_i(n_i-k_i)$ and are multiplied by $\alpha_0!\cdots \alpha_K!$. A first step to compute $g(\bfk,\bfn)$, is to construct the flattenings of both~$A$ and $B$.

The function \texttt{flatMat} computes the flattenings of a general tensor in a given set of variables. More precisely, given as input a variable \texttt{a} and a tensor format $\bfk=(k_1,\dots,k_d)$, \texttt{flatMat(a,k)} computes all flattenings $T^{(1)},\dots,T^{(d)}$ of a general tensor $T\in \CC^{k_1}\otimes \cdots \otimes \CC^{k_d}$ whose coordinates are
$
T_{i_1,\dots,i_d}:=a_{i_1,\dots,i_d}.
$

{\small
\begin{lstlisting}[language=code]
flatMat = (a,k) ->(
    ringC := QQ[a_(length k:1)..a_k]; 
    ringT = tensor(apply(1..length k,i->
	    (ringC[x_(i,1)..x_(i,k_(i-1))]))); 
    ones := toList (length k:1);  
    coeffList := apply(flatten entries vars ringC,i->sub(i,ringT));
    varsList := entries  tensor(apply(1..length k,i->
	    vector(toList(x_(i,1)..x_(i,k_(i-1))))));
    genericT := sum toList apply(1..length varsList ,i->(
	    coeffList_(i-1))*(varsList_(i-1)));
    return toList apply(1..length k,i->  
        contract(basis(ones -flatten entries (id_(ZZ^(length k)))_(i-1) ,ringT),
        transpose contract(basis(flatten entries (id_(ZZ^(length k)))_(i-1),ringT),
        genericT)))   
    )
\end{lstlisting}
}

For instance, the output of \texttt{flatMat(w,(2,2,2))} is a list containing the following three matrices:
$$
\begin{pmatrix}
      w_{1,1,1}&w_{1,1,2}&w_{1,2,1}&w_{1,2,2}\\
      w_{2,1,1}&w_{2,1,2}&w_{2,2,1}&w_{2,2,2}\\
      \end{pmatrix}, \quad 
      \begin{pmatrix}
      w_{1,1,1}&w_{1,1,2}&w_{2,1,1}&w_{2,1,2}\\
      w_{1,2,1}&w_{1,2,2}&w_{2,2,1}&w_{2,2,2}\\
      \end{pmatrix}, \quad 
      \begin{pmatrix}
      w_{1,1,1}&w_{1,2,1}&w_{2,1,1}&w_{2,2,1}\\
      w_{1,1,2}&w_{1,2,2}&w_{2,1,2}&w_{2,2,2}\\
      \end{pmatrix}.
$$

We will use the function \texttt{flatMat} to compute the matrices of the flattenings in the two different sets of variables~$a_{i_1,\dots,i_d}$ and $b_{i_1,\dots,i_d}$. 

The next function \texttt{prodPowerDetList} takes as input two sequences $\bfk=(k_1,\dots,k_d)$ and $\bfn=(n_1,\dots,n_d)$ and returns the terms of the polynomial $\prod_{\ell=1}^d \det(A^{(\ell)}(B^{(\ell)})^T)^{n_\ell-k_\ell}$.

{\small
\begin{lstlisting}[language=code]
prodPowerDetList = (k,n)->(
    K := product toList k;
    R1 := QQ[a_(length k:1)..a_k,Degrees =>entries(id_(ZZ^K))];
    R2 := QQ[b_(length k:1)..b_k,Degrees =>entries(id_(ZZ^K))];
    finalR := R1**R2; 
    flatList1 := apply(flatMat(a,k),i->sub(i,finalR));
    flatList2 := apply(flatMat(b,k),i->sub(i,finalR));
    detList := apply(1..length k, i->(    
	det((flatList1_(i-1))*(transpose flatList2_(i-1)),Strategy=>Cofactor)
	) );
    return terms(product (toList(apply(1..length k,
    i->detList_(i-1)^(n_(i-1)-k_(i-1))))))
 )
\end{lstlisting} 
}

We are now ready to construct a final function that computes the degree of a subspace variety $X_{\bfk}\subset \PP^N$, once taken as input the sequences $\bfk=(k_1,\dots,k_d)$ and $\bfn=(n_1,\dots,n_d)$. Our first step is to define a list \texttt{L} containing the terms of the polynomial $\prod_{\ell=1}^d \det(A^{(\ell)}(B^{(\ell)})^T)^{n_\ell-k_\ell}$. We then collect in \texttt{sameBdegList} all terms of \texttt{L} that have the same bi-degree in the variables $a_{i_1,\dots,i_d}$, $b_{i_1,\dots,i_d}$. The list \texttt{finalSum} contains all terms that compose the sum $g(\bfk,\bfn)=(k_1(n_1-k_1))! \cdots (k_d(n_d-k_d))!\cdot f(\bfk,\bfn)$. Having now all ingredients constituting the formula of the main \Cref{main}, we are able to return $\deg(X_{\bfk})$.

{\small
\begin{lstlisting}[language=code]
degreeSubspace = (k,n)->(
    L := prodPowerDetList(k,n);
    K := product toList k;
    sameBdegList := L_(
    	 positions(L,i-> 
	    (degree i)_{0..K-1} == (degree i)_{K..(length degree i-1)}));	  
    finalSum := apply(sameBdegList,i->(
    	 de := (degree i)_{0..K-1};
     	 molt := product apply(de,x->x!);
             (coefficients(i))_1_(0,0)*molt)
	 );
    prodDegG := product apply(toList(1..length k),z->degG(k_(z-1),n_(z-1)) );
    prodDimG := product toList apply(1..length k,i->(k_(i-1)*(n_(i-1)-k_(i-1)))!); 
    return (sub(sum finalSum,ZZ))*prodDegG/prodDimG 
    ) 
\end{lstlisting}    
}
For instance, the first entry in the table in \Cref{ex2} is computed by the following commands:   {\small
\begin{lstlisting}[language=code]
k = (1,2,2); n = (3,3,3)
degreeSubspace(k,n)
\end{lstlisting}
}

\bigskip
\bigskip

\end{document}